\documentclass[12pt,a4paper]{article}

\usepackage{latexsym,amssymb,cite}

\newtheorem{theorem}{Theorem}
\newtheorem{lemma}[theorem]{Lemma}

\newenvironment{proof}{
\par
\noindent {\bf Proof.}\rm}%
{\mbox{}\hfill\rule{0.5em}{0.809em}\par}

\begin{document}

\baselineskip=19pt
\parindent=0.5cm

\title{\bf Nordhaus-Guddam Type Relations of\\ 
Three Graph Coloring Parameters}

\author{
Kuo-Ching Huang\\
\normalsize  Department of Financial and Computational Mathematics\\
\normalsize  Providence University\\
\normalsize  Taichung 43301, Taiwan\\
\normalsize {\tt Email: kchuang@gm.pu.edu.tw}
\and
Ko-Wei Lih\thanks{Supported in part by the
National Science Council under grant  NSC99-2115-M-001-004-MY3}\\
\normalsize Institute of Mathematics\\
\normalsize Academia Sinica\\
\normalsize Taipei 10617, Taiwan\\
\normalsize {\tt Email: makwlih@sinica.edu.tw}
}

\date{\small }

\maketitle

%
\begin{abstract}
%

\noindent
Let $G$ be a simple graph. A coloring of vertices of $G$ is called 
(i) a $2$-proper coloring if vertices at distance 2 receive distinct 
colors; (ii) an injective coloring if vertices possessing a common 
neighbor receive distinct colors; (iii) a square coloring if vertices 
at distance at most 2 receive distinct colors. In this paper, we study 
inequalities of Nordhaus-Guddam type for the $2$-proper chromatic number, 
the injective chromatic number, and the square chromatic number. 

\bigskip

\end{abstract}

%
\section{Introduction}
%

Let $G=(V,E)$ be a finite  simple graph with vertex set $V(G)$ and edge 
set $E(G)$. The {\em order} $|G|$ of $G$ is the cardinality of $V(G)$. 
The {\em degree} $d_G(v)$ of a vertex $v\in V(G)$ is the number of edges 
incident to $v$. The {\em maximum} and {\em minimum} degree of $G$ are 
denoted by $\Delta(G)$ and $\delta(G)$, respectively. The {\em neighborhood} 
$N_G(v)$ of a vertex $v\in V(G)$ is the set of vertices adjacent to $v$. 
The {\em distance} $d_G(u,v)$ between two vertices $u$ and $v$ is the 
length of a shortest $(u,v)$-path. We abbreviate $d_G(u,v)$ to $d(u,v)$ 
when no ambiguity arises. A subset $S$ of $V(G)$ is an {\em independent} 
set of $G$ if $uv\not \in E(G)$ for all vertices $u$ and $v$ in $S$. A 
subset $W$ of $V(G)$ is a {\em clique} of $G$ if $uv\in E(G)$ for all 
vertices $u$ and $v$ in $W$. A clique on $n$ vertices is denoted by $K_n$. 
The {\em complement} $\overline{G}$ of $G$ is the graph defined on the 
vertex set $V(G)$ of $G$ such that an edge $uv\in E(\overline{G})$ if and 
only if $uv\not \in E(G)$. 

Let $k$ be a positive integer. A mapping $f:V(G)\rightarrow \{1,2,\ldots ,
k\}$ is called a (proper) $k$-{\em coloring} of $G$ if $f(u)\neq f(v)$ 
whenever $uv\in E(G)$. The {\em chromatic number} $\chi(G)$ of $G$ is the 
minimum number $k$ such that $G$ has a $k$-coloring. The following is a 
well-known theorem of Nordhaus and Guddam \cite{ng}. 

\begin{theorem}
If $G$ is a graph of order $n$, then
\begin{enumerate}
\item
$2\sqrt{n}\leqslant \chi(G)+\chi(\overline{G})\leqslant n+1.$ 
\item
$n\leqslant \chi(G)\chi(\overline{G}) \leqslant (n+1)^2/4.$
\end{enumerate}
\end{theorem}

\bigskip

Inequalities involving the sum or product of a parameter applied to a graph 
and its complement are commonly known as Nordhaus-Guddam type relations. 
The reader is referred to Aouchiche and Hansen \cite{ah} for a recent survey.

A mapping $f:V(G)\rightarrow \{1,2,\ldots , k\}$ is called
\begin{itemize}
\item
a {\em $2$-proper $k$-coloring} of $G$ if $f(u)\neq f(v)$ 
whenever $d(u,v)=2$; 
\item
an {\em injective $k$-coloring} of $G$ if $f(u)\neq f(v)$ 
whenever the $u$ and 
$v$ have a common neighbor;
\item
a {\em square $k$-coloring} of $G$ if $f(u)\neq f(v)$ 
whenever $d(u,v)\leqslant 2$.
\end{itemize}

The minimum number $k$ such that $G$ has a $2$-proper, an injective, or a 
square $k$-coloring is called the {\em $2$-proper, injective, {\em or} 
square chromatic number of $G$}.  They are denoted by $\chi_2(G)$, $\chi_i(G)$, 
and $\chi_{_{\square}}(G)$, respectively. Let $G^2$ be the square graph of $G$ 
obtained by adding a new edge between any pair of vertices that are distance 2 
apart in $G$. Obviously, $\chi_{_{\square}}(G)$ is precisely $\chi(G^2)$.  

The above graph colorings are closely related to a more general notion of graph 
labelings. Let $p$ and $q$ be two nonnegative integers. A $k$-$L(p,q)$-labeling 
of a graph $G$ is a mapping $f: V(G)\rightarrow \{0, 1, \ldots , k\}$ such that 
$|f(u)-f(v)|$ is at least $p$ if $d(u,v)=1$ and at least $q$ if $d(u,v)=2$. The 
$L(p,q)$-labeling number $\lambda(G;p,q)$ of $G$ is the least $k$ such that $G$ 
has a $k$-$L(p,q)$-labeling with $\max\{f(v) \mid v \in V(G)\} = k$. Obviously, 
an $L(1,0)$-labeling of a graph $G$ is a proper coloring of $G$ and $\chi(G)= 
\lambda(G;1,0)+1$; an $L(0,1)$-labeling of a graph $G$ is a 2-proper coloring of 
$G$ and $\chi_2(G)= \lambda(G;0,1)+1$; an $L(1,1)$-labeling is a square coloring 
and $\chi_{_{\square}}(G)= \lambda(G;1,1)+1$. Note that, if $G$ is triangle-free, 
then $\chi_i(G)=\chi_2(G)$. The reader is referred to Yeh \cite{yeh} for a survey 
on $L(p,q)$-labelings of graphs. The injective coloring has been studied in
\cite{bi,bcrw,cky10,cky11,lst}.

In this paper, we study inequalities of Nordhaus-Guddam type for the $2$-proper 
chromatic number, the injective chromatic number, and the square chromatic number. 
Graphs attaining extrema are also obtained.

%
\section{2-proper chromatic numbers}
%

For a given coloring of a graph, a {\em color class} consists of all vertices 
of a fixed color. Note that any color class of a $2$-proper coloring consists 
of disjoint cliques. For $n_1\geqslant n_2\geqslant \cdots \geqslant n_r 
\geqslant 1$, let $K_{n_1,n_2, \ldots , n_r}$ denote the {\em complete 
$r$-partite graph} such that its vertex set has $r$ disjoint parts with edges 
joining every pair of vertices belonging to different parts.

\begin{lemma}\label{2}
For $r\geqslant 2$,  $\chi_2(K_{n_1,n_2, \ldots , n_r})=n_1$.
\end{lemma}

\begin{proof}
Let $\{V_1,V_2,\ldots , V_r\}$ denote the parts of $G=K_{n_1,n_2, \ldots , 
n_r}$ with $|V_i|=n_i$, $1\leqslant i\leqslant r$. For each $i$, color the 
vertices of $V_i$ with colors $1, 2, \ldots , n_i$ such that no pair of 
vertices receiving the same color to obtain a $2$-proper coloring. Hence 
$\chi_2(G)\leqslant n_1$. Since any two vertices in $V_1$ are at distance 2, 
$\chi_2(G)\geqslant n_1$. 
\end{proof}

\begin{theorem}
For any graph $G$ of order $n$,
\[
 1\leqslant (\chi_2(G)\chi_2(\overline{G}))^{1/2} \leqslant 
 \frac{\chi_2(G)+\chi_2(\overline{G})}{2} \leqslant  \frac{n+1}{2}.
\]
\end{theorem}

\begin{proof}
It suffices to prove that $\chi_2(G)+\chi_2(\overline{G})\leqslant n+1.$
Without loss of generality, we may suppose that $\chi_2(G)\geqslant 
\chi_2(\overline{G}).$ If $\chi_2(G) \leqslant (n+1)/2$, then $\chi_2(G)+
\chi_2(\overline{G})\leqslant n+1$. Now assume that $\chi_2(G)> (n+1)/2$. 

Among all $2$-proper colorings of $G$ using $\chi_2(G)$ colors, let $f$ be 
chosen with the maximum number of singleton color classes. Let $\{X_i\}_{i=1}^a$,
$\{Y_j\}_{j=1}^b$, and $\{Z_k\}_{k=1}^c$ denote, respectively, the collections 
of color classes of $f$ such that each $X_i$ is a singleton, each $Y_j$ consists 
of a single clique of size at least two, and each $Z_k$ consists of at least two 
disjoint cliques. Thus $\chi_2(G)=a+b+c> (n+1)/2$. First note that $a > 0$, for 
otherwise $n\geqslant 2b+2c= 2\chi_2(G) > n + 1$.

Let $\mathcal{X}=\bigcup_{i=1}^a X_i$, $\mathcal{Y}=\bigcup_{j=1}^b Y_j$, and
$\mathcal{Z}=\bigcup_{k=1}^c Z_k$. Then $\mathcal{X}$ must be an independent 
set, for otherwise we may re-color two adjacent vertices in $\mathcal{X}$ with 
the same color to obtain a $2$-proper coloring of $G$ using $\chi_2(G)-1$ colors. 
The complement $\overline{G[Z_k]}$ of the subgraph $G[Z_k]$ induced by $Z_k$ in 
$G$ is a complete multipartite graph. By Lemma \ref{2}, $\chi_2(\overline{G[Z_k]})
\leqslant |Z_k|-1$.

Now suppose that $b>0$. There is a vertex $u_1$ of $Y_1$ that is non-adjacent 
to any vertex in $\mathcal{X}$. Otherwise, we could re-color each vertex $y$ of 
$Y_1$ with color $f(x_{i_y})$, where $i_y=\min\{t \mid x_t\in N_G(y)\cap 
\mathcal{X}\}$, to obtain a $2$-proper coloring of $G$ with $\chi_2(G)-1$ colors. 
Next, we move any vertex $v \in Y_1$ that is different from $u_1$ and adjacent 
to all vertices of $Y_2$ from $Y_1$ to $Y_2$. In view of the maximality of $a$, 
we are left with at least one $v_1 \in Y_1$ that is different from $u_1$ and 
non-adjacent to a certain vertex $u_2\in Y_2$ if $b>1$. We may repeat this 
process of moving vertices to the next color class until we obtain a sequence 
of vertices $u_1,v_1,u_2,v_2, \ldots , u_b,v_b$ such that $u_j, v_j\in Y_j$, 
and $u_j \ne v_j$ for $1\leqslant j\leqslant b$ and $v_ju_{j+1}\in E(\overline{G})$ 
for $1\leqslant j\leqslant b-1$. Now, in $\overline{G}$, we color $u_1$ and the 
vertices in $\mathcal{X}$ with color 1, $v_j$ and $u_{j+1}$ with color $j+1$ for 
$1\leqslant j\leqslant b-1$, and the vertices in $\mathcal{Y} \setminus \{u_1,v_1, 
\ldots , v_{b-1},u_b\}$ with colors $b+1,b+2,\ldots , |\mathcal{Y}| -b+1$ such that 
no pair of vertices receiving the same color. It follows that
\[
\begin{array}{rcl}
 \chi_2(\overline{G})& \leqslant & \sum_{k=1}^c \chi_2(\overline{G[Z_k]})+ 
 												|\mathcal{Y}| -b+1\\
                     & \leqslant & |\mathcal{Z}|-c+ |\mathcal{Y}|-b+1\\
                     & =         & n-a-b-c+1\\
                     & =         & n-\chi_2(G)+1.\\
\end{array}
\]
The above inequalities hold even if $b=0$. Therefore, $\chi_2(G)+\chi_2(\overline{G})
\leqslant n+1$.       
\end{proof}

\bigskip

Let us consider the sharpness of inequalities in the above theorem. The lower 
bound is sharp since $\chi_2(K_n)=\chi_2(\overline{K_n})=1$. For the case of 
upper bound, we first construct an auxiliary graph $H_k$ as follows. Let $k 
\geqslant 6$. The vertex set of $H_k$ can be partitioned into an independent 
set $X=\{ x_0, x_1, \ldots , x_{k-1} \}$ and a clique $Y=\{ y_0, y_1, \ldots , 
y_{k-1} \}$ so that each $x_i$ is joined to $y_i, y_{i+1}, \ldots , y_{i+\lfloor 
k/2 \rfloor}$ except $y_{i+\lfloor k/2 \rfloor - 1}$. Here indices are taken 
modulo $k$.

For $0 \leqslant i \ne j < k$, if $y_{i+\lfloor k/2 \rfloor - 1}$ and $y_{j+
\lfloor k/2 \rfloor - 1}$ are not neighbors of both $x_i$ and $x_j$, then 
$x_i$ and $x_j$ together have $2 \lfloor k/2 \rfloor > k-2$ edges joining $Y$. 
Hence, they must have a common neighbor and $d_{H_k}(x_i,x_j)=2$. Suppose that 
$x_i$ is adjacent to $y_{j+\lfloor k/2 \rfloor - 1}$. Since $k \geqslant 6$, 
there are three possibilities: (i) $y_{i+\lfloor k/2 \rfloor}= y_{j+\lfloor k/2 
\rfloor -1}$; (ii) $y_{i+\lfloor k/2 \rfloor-2}=y_{j+\lfloor k/2 \rfloor -1}$; 
(iii) $y_{i+t}=y_{j+\lfloor k/2 \rfloor -1}$ for some $0 \leqslant t \leqslant 
\lfloor k/2 \rfloor-3$. Then $x_i$ and $x_j$ have a common neighbor $z$, where
$z$ is $y_j$ for (i), $y_{j+1}$ for (ii), and $y_{j+\lfloor k/2 \rfloor}$ for 
(iii). Again, $d_{H_k}(x_i,x_j)=2$.

The complement graph $\overline{H_k}$ can be isomorphically described as
follows. Let $X=\{ x_0, x_1, \ldots , x_{k-1} \}$ be a clique and $Y=\{ y_0, 
y_1, \ldots , y_{k-1} \}$ be an independent set such that each $y_i$ is joined 
to $x_i, x_{i+1}, \ldots , x_{i+\lceil k/2 \rceil}$ except $x_{i+\lceil k/2 
\rceil-1}$. When $k$ is even, $\overline{H_k}$ is isomorphic to $H_k$. When $k$ 
is odd, any $y_i$ and $y_j$, $i \ne j$, together have $2 \lceil k/2 \rceil = k+1$ 
edges joining $X$. It follows that $d_{\overline{H_k}}(y_i,y_j)=2$ for $0 \leqslant 
i \ne j < k$. 

In the second step, we construct a graph $H_{\rm od}$ of order $2k+1 \geqslant 
13$ and a graph $H_{\rm ev}$ of order $2k+2 \geqslant 14$ as follows. We join 
a new vertex $\infty$ to all $y_i$'s in $H_k$ to obtain $H_{\rm od}$ and two 
new independent vertices $\infty_1$ and $\infty_2$ to all $y_i$'s in $H_k$ to 
obtain $H_{\rm ev}$. It is straightforward to see that $\chi_2(H_{\rm od}) +
\chi_2(\overline{H_{\rm od}})=2k+2$ and $\chi_2(H_{\rm ev})+
\chi_2(\overline{H_{\rm ev}})=2k+3$.

%
\section{Injective chromatic numbers}
%

For the injective chromatic number $\chi_i(G)$ of a graph $G$, it is clear 
that $\Delta(G)\leqslant \chi_i(G)\leqslant |G|$. Note that if $S$ is a 
color class of an injective $k$-coloring, then $\Delta(G[S])\leqslant 1$.

Suppose $G$ is a graph of order $n\leqslant 4$. It is routine to check that 
(i)\ $n\leqslant \chi_i(G)+\chi_i(\overline{G}) \leqslant 2n$ except $\chi_i(C_4)
+\chi_i(\overline{C_4})=3$; (ii)\ $n\leqslant \chi_i(G) \chi_i(\overline{G})
\leqslant n^2$ except $G \in \{ K_2$, $\overline{K_2}$, $P_3$, $\overline{P_3}$, 
$C_4$, $\overline{C_4} \}$. Here, $P_n$ and $C_n$ denote a path and a cycle on 
$n$ vertices, respectively.

\begin{lemma}\label{4}
Suppose that the graph $G$ has order $n \geqslant 5$. Then the following 
statements hold.

{\rm (1)}\ 
If $\delta(G)\geqslant (n+1)/2$, then $\chi_i(G)=n.$

{\rm (2)}\
If $\delta(G)= \lfloor (n-1)/2 \rfloor$, then $\chi_i(G)
\geqslant \delta(G)+1$.
\end{lemma}

\begin{proof}
(1) Since $\delta(G)\geqslant (n+1)/2$, $d_G(u) + d_G(v)\geqslant n+1$ 
for any two vertices $u$ and $v$ in $G$. Then $u$ and $v$ have a common 
neighbor. Hence, $\chi_i(G)=n.$

(2) If $\Delta(G) > \delta(G)$, then $\chi_i(G)\geqslant \Delta(G) \geqslant 
\delta(G)+1$. Consider $\Delta(G) = \delta(G)= \lfloor (n-1)/2 \rfloor = k$. 
Suppose $\chi_i(G)=k$ and let $\{V_1, V_2, \dots , V_k\}$ be the set of color 
classes of an injective $k$-coloring of $G$. If $|V_i|\leqslant 2$ for all $i$, 
then $n=\sum_{i=1}^{k}|V_i| \leqslant 2k\leqslant n-1$, a contradiction. Assume 
that, for some $i$, $V_i$ contains at least three vertices $v_1, v_2, v_3$. 
Since no two vertices in $V_i$ have a common neighbor, $\Delta(G[V_i])\leqslant 
1$ and hence $n-3 \geqslant |\bigcup_{i=1}^3N_G(v_i) \setminus \bigcup_{i=1}^3 
\{v_i\}| \geqslant 2(k-1)+k>n-3$ when $n \geqslant 5$, again a contradiction.
\end{proof}

\begin{lemma}\label{5}
Suppose $G$ is a $k$-regular graph of order $n\geqslant 5$.  

{\rm (1)}\
If $k > n/2$ or $k < (n-2)/2$, then $n+1\leqslant \chi_i(G)+ \chi_i(\overline{G})
\leqslant 2n$.

{\rm (2)}\
If $k = n/2$ or $(n-2)/2$, then $n\leqslant \chi_i(G)+ \chi_i(\overline{G})
\leqslant 2n$.
\end{lemma}

\begin{proof}
The upper bounds are obvious. Note that, since $G$ is $k$-regular, $\overline{G}$ 
is $k'$-regular, where $k' = n-k-1$.  

(1) If $k > n/2$, by (1) of Lemma \ref{4}, $\chi_i(G)=n.$ Then $\chi_i(G)+
\chi_i(\overline{G})= n+ \chi_i(\overline{G})\geqslant n+1$. If $k < (n-2)/2$, 
then $k' > n/2$.  By (1) of Lemma \ref{4}, $\chi_i(\overline{G})=n$ and then
$\chi_i(G)+\chi_i(\overline{G})= \chi_i(G)+ n\geqslant n+1$.

(2) If $k = n/2$, then $k' = (n-2)/2$. By (2) of Lemma \ref{4}, $\chi_i(\overline{G})
\geqslant k' + 1$ and then $\chi_i(G)+ \chi_i(\overline{G}) \geqslant k+k'+1= n$. 
If $k = (n-2)/2$, by (2) of Lemma \ref{4}, $\chi_i(G)\geqslant k + 1$ and then 
$\chi_i(G) +\chi_i(\overline{G})\geqslant k+1+k' = n$.  
\end{proof}

\begin{theorem}
Suppose $G$ is a graph of order $n\geqslant 5$.

{\rm (1)}\  
If $n=5$ or $n$ is even, then  $n\leqslant \chi_i(G)+ \chi_i(\overline{G})
\leqslant 2n.$

{\rm (2)}\
If $n\geqslant 7$ is odd, then  $n+1\leqslant \chi_i(G)+ \chi_i(\overline{G})
\leqslant 2n.$
\end{theorem}

\begin{proof}
The upper bounds are obvious. It is clear that $\chi_i(G)+ \chi_i(\overline{G})
\geqslant\Delta(G)+\Delta(\overline{G})=\Delta(G)-\delta(G)+n-1\geqslant n+1$ 
if $\Delta(G)-\delta(G) \geqslant 2$.

{\em Case 1.}\  $\Delta(G)=\delta(G)=k$.

Then $G$ is $k$-regular and $\overline{G}$ is $k'$-regular, where $k' = n-k-1$. 
If $n$ is even, then Lemma \ref{5} implies $\chi_i(G)+\chi_i(\overline{G})
\geqslant n$. Moreover, suppose that $n$ is odd. If $k\neq (n-1)/2$, $\chi_i(G)
+\chi_i(\overline{G}) \geqslant n+1$ by (1) of Lemma \ref{5}. If $k = (n-1)/2$, 
then $k' = (n-1)/2$. By (2) of Lemma \ref{4}, $\chi_i(G)+ \chi_i(\overline{G})
\geqslant k+1+k'+1= n+1$.

{\em Case 2.}\  $\Delta(G)-\delta(G)= 1$.

Then $\chi_i(G)+\chi_i(\overline{G})\geqslant \Delta(G)+\Delta(\overline{G}) 
\geqslant n$ and (1) is established. Next, let $n\geqslant 7$ be an odd integer. 
Suppose on the contrary that $\chi_i(G)+\chi_i(\overline{G})= n$. Then $\chi_i(G)
=\Delta(G)$, $\chi_i(\overline{G})=\Delta(\overline{G})$ and $\Delta(G)+ 
\Delta(\overline{G})=n$. Without loss of generality, we may assume $\Delta(G)
\geqslant \Delta(\overline{G})$. Hence, $\Delta(G) \geqslant n/2$ which implies 
$\Delta(G) \geqslant (n+1)/2$. If $\Delta(G)>(n+1)/2$, then $\delta(G)=\Delta(G)-1
\geqslant (n+1)/2$. By (1) of Lemma \ref{4}, $\chi_i(G)=n$ and then $\chi_i(G)+
\chi_i(\overline{G})\geqslant n+1$. Suppose $\chi_i(G)=\Delta(G)=(n+1)/2$. Then
$\chi_i(\overline{G})=\Delta(\overline{G})=(n-1)/2$ and $\delta(\overline{G})=
(n-3)/2$. Let $p=(n-1)/2$ and $\{V_1, V_2, \dots , V_p\}$ be the set of color 
classes of an injective $p$-coloring of $\overline{G}$. Since $p=(n-1)/2$, $V_i$ 
contains at least three vertices $v_1, v_2, v_3$ for some $i$. Since no two 
vertices in $V_i$ have a common neighbor, $\Delta(G[V_i])\leqslant 1$ and hence 
$n-3 \geqslant |\bigcup_{i=1}^3N_G(v_i) \setminus \bigcup_{i=1}^3 \{v_i\}| 
\geqslant 2(\delta(\overline{G})-1)+\delta(\overline{G})=(n-3)+(n-7)/2$. It 
follows that $n=7$ and $N_{\overline{G}}(v_1)=\{v_2,v_4\}$, $N_{\overline{G}}(v_2)
=\{v_1,v_5\}$ and $N_{\overline{G}}(v_3)=\{v_6,v_7\}$. Then, in $G$, we have 
$v_1v_5\in E(G)$, $v_2v_4\in E(G)$, and $N_G(v_3)=\{v_1,v_2,v_4,v_5\}$. Therefore, 
any pair $v_i$ and $v_j$, $1\leqslant i < j\leqslant 5$, have a common neighbor. 
Then $5 \leqslant \chi_i(G)=(n+1)/2=4$, a contradiction. Therefore, $\chi_i(G)
+\chi_i(\overline{G})\geqslant n+1$.
\end{proof}

\begin{theorem}
For any graph $G$ of order $n\geqslant 5$, $n\leqslant \chi_i(G)\chi_i(\overline{G})
\leqslant n^2.$
\end{theorem}

\begin{proof}
The upper bound is obvious. Let $G$ be a graph of order $n\geqslant 5$. Suppose 
$\chi_i(G)=p$ and $\chi_i(\overline{G})=q$. Let $f$ and $g$ be injective $p$-coloring 
and $q$-coloring of $G$ and $\overline{G}$, respectively. Define a mapping $h: V(K_n)
\rightarrow \{1, 2, \ldots , p\}\times \{1, 2, \ldots , q\}$ by $h(u)=(f(u),g(u))$ 
for all $u\in V(K_n)$. If $h(u)\neq h(v)$ for all vertices $u$ and $v$, then $h$ is 
an injective $pq$-coloring of $K_n$. Hence, $n=\chi_i(K_n)\leqslant pq =\chi_i(G)
\chi_i(\overline{G})$. Suppose $h(u)= h(v)$ for some vertices $u$ and $v$. Without 
loss of generality, we may assume $uv\in E(G)$. Since $f(u)= f(v)$, $N_G(u) \cap
N_G(v)=\emptyset$. Since $g(u)= g(v)$, $x\in N_G(u)\cup N_G(v)$ for all vertices $x$ 
in $G$. Then $N_G(u)\cup N_G(v)=V(G)$ and any vertex $x$ in $G$ is adjacent to exact 
one of $u$ and $v$. Suppose $d_G(u)=a\geqslant d_G(v)=n-a$. Then $\chi_i(G)
\chi_i(\overline{G}) \geqslant d_G(u)d_{\overline{G}}(v)=a(a-1)\geqslant {\left \lceil 
n/2 \right \rceil}({\left \lceil n/2\right \rceil}-1)\geqslant n$.  
\end{proof}

\bigskip 

Consider the sharpness of the lower bounds. For $n=5$, $\chi_i(P_5)+
\chi_i(\overline{P_5})=5$. For $n=2k\geqslant 6$, $\chi_i(K_{k,k})+
\chi_i(\overline{K_{k,k}})=k+k=n$. For $n=2k+1\geqslant 7$, $\chi_i(K_{k+1,k})+
\chi_i(\overline{K_{k+1,k}})=k+1+k+1=n+1$. For $n\neq 2$, $\chi_i(K_n)
\chi_i(\overline{K_n})=n$.

Now consider the sharpness of the upper bounds. Note that $\chi_i(G)=|G|$ 
if and only if any two distinct vertices in $G$ have a common neighbor. Using 
this fact, we may see that $\chi_i(G)+\chi_i(\overline{G})<2|G|$ if $1<|G|<9$.
For $k\geqslant 3$, let $n=3k+t\geqslant 9$, where $t=0,1$ or $2$. We construct 
an auxiliary graph $G_{3k}$ as follows. The vertex set of $G_{3k}$ can be 
partitioned into three cliques $X=\{ x_1, x_2, \ldots , x_k \},~Y=\{ y_1, y_2, 
\ldots , y_k \}$, and $Z=\{ z_1, z_2, \ldots , z_k \}$ such that $\{x_i, y_i, z_i \}$ 
forms a clique for all $1\leqslant i\leqslant k$. We join a new vertex $\infty$ 
to all $x_i$'s in $G_{3k}$ to obtain $G_{3k+1}$ and two new vertices $\infty_1$ 
and $\infty_2$ to all $x_i$'s in $G_{3k}$ to obtain $G_{3k+2}$. It can be verified 
that any two distinct vertices in $G_n$ and $\overline{G_n}$ have a common neighbor. 
Hence, $\chi_i(G_n)+\chi_i(\overline{G_n})=n+n=2n$ and $\chi_i(G_n)
\chi_i(\overline{G_n})=n^2$.

%
\section{Square chromatic numbers}
%

Since any pair of vertices that are adjacent or distance 2 apart receive 
distinct colors in a square coloring, every color class of a square coloring 
must be an independent set. 

\begin{theorem}
For any graph $G$ of order $n$, $n+1\leqslant \chi_{_{\square}}(G)+ 
\chi_{_{\square}}(\overline{G}) \leqslant 2n$, or equivalently, $n+1
\leqslant \chi(G^2)+\chi(\overline{G}^2)\leqslant 2n$.
\end{theorem}

\begin{proof}
The upper bound is obvious. Suppose $G$ is a graph of order $n$ and 
$\chi(G^2)=p$. Let $f=(X_1,\dots, X_a,Y_1,\dots,Y_b)$ be a square 
$p$-coloring of $G$ with $a+b=p$, $|X_i|=1$ and $|Y_j| \geqslant 2$ 
for all $i$ and $j$. If $a=p$, then $a=n$ and $\chi(G^2)+\chi(\overline{G}^2)
= n+\chi(\overline{G}^2)\geqslant n+1$. Suppose $a<p$. Since $f$ is a square 
coloring, each $Y_j$ is an independent set of $G$ and any vertex $u$ in $Y_i$ 
has at most one neighbor in $Y_j$ for all $i\neq j$. Hence, $uv\in 
E(\overline{G})$ for some $v$ in $Y_j$. Then $d_{\overline{G}}(u,v)
\leqslant 2$ for all vertices $u$ and $v$ in $\bigcup_{j=1}^{b}Y_j$. 
Therefore, $\chi(\overline{G}^2)\geqslant n-a\geqslant n-p+1=n-
\chi(G^2)+1$, or $\chi(G^2)+\chi(\overline{G}^2)\geqslant n+1$.
\end{proof}

\begin{theorem}
For any graph $G$ of order $n$, $n\leqslant \chi_{_{\square}}(G) 
\chi_{_{\square}}(\overline{G}) \leqslant n^2$, or equivalently, 
$n\leqslant \chi(G^2)\chi(\overline{G}^2) \leqslant n^2$.
\end{theorem}

\begin{proof}
The upper bound is obvious. Since $\chi(G) \leqslant \chi(G^2)$ and 
$\chi(\overline{G}) \leqslant \chi(\overline{G}^2)$, the lower bound 
is a consequence of the Nordhaus-Guddam theorem.
\end{proof}

\bigskip

Consider the sharpness of the lower bounds. It is clear that $\chi(K_n^2)=n$ 
and $\chi(\overline{K_n}^2)=1$. Hence, $\chi(K_n^2)+\chi(\overline{K_n}^2)=n+1$  
and $\chi(K_n^2)\chi(\overline{K_n}^2)=n$.
 
Now consider the sharpness of the upper bounds. For $2\leqslant n \leqslant 4$, 
it is routine to check that $\chi(G^2)+\chi(\overline{G}^2) \leqslant 2n-1$ if 
$|G|=n$. For $n\geqslant 5$, we construct a graph $F_n$ of order $n$ as follows. 
The vertex set of $F_n$ can be partitioned into a 5-cycle  $C_5=x_1x_2x_3x_4x_5x_1$ 
and an independent set $Y=\{ y_1, y_2, \ldots , y_{n-5} \}$ such that each $y_i$ is 
adjacent to both $x_1$ and $x_3$. It can be verified that any two vertices in $F_n$, 
or $\overline{F_n}$, are at distance at most 2. Hence, $\chi(F_n^2)+
\chi(\overline{F_n}^2)=n+n = 2n$ and $\chi(F_n^2)\chi(\overline{F_n}^2)=n^2$.

\bigskip

%

\end{document}